\documentclass[reqno]{amsart}

\usepackage[backend=bibtex, sorting=none, sorting=nyt,maxbibnames=99]{biblatex}

\renewbibmacro{in:}{}
\bibliography{references.bib}
\usepackage{amssymb}
\usepackage{calrsfs}
\usepackage{graphics,graphicx, mathrsfs}
\usepackage{enumerate}
\usepackage{enumitem}
\usepackage{url}
\usepackage{xcolor}
\definecolor{vio}{rgb}{0.54, 0.17, 0.89}
\usepackage{halloweenmath}
\usepackage[ruled, lined, linesnumbered, longend]{algorithm2e}
\usepackage{hyperref}
\usepackage[justification=centering]{caption}

\newtheorem{theorem}{Theorem}[section]
\newtheorem{lemma}[theorem]{Lemma}

\newtheorem{corollary}[theorem]{Corollary}
\numberwithin{equation}{section}

\theoremstyle{remark}
\newtheorem*{remark}{Remark}

\def\reals{\hbox{\rm I\kern-.18em R}}
\def\complexes{\hbox{\rm C\kern-.43em
\vrule depth 0ex height 1.4ex width .05em\kern.41em}}
\def\field{\hbox{\rm I\kern-.18em F}} 

\allowdisplaybreaks

\begin{document}

\title[The error term in the explicit formula of Riemann--von Mangoldt]{On the error term in the explicit formula of Riemann--von Mangoldt}

\author{Michaela Cully-Hugill and Daniel R. Johnston}
\address{School of Science, UNSW Canberra, Australia}
\email{m.cully-hugill@adfa.edu.au}
\address{School of Science, UNSW Canberra, Australia}
\email{daniel.johnston@adfa.edu.au}
\date\today
\keywords{}

\begin{abstract}
    We provide an explicit $O(x\log x/T)$ error term for the Riemann--von Mangoldt formula by making results of Wolke (1983) and Ramar\'e (2016) explicit. We also include applications to primes between consecutive powers, the error term in the prime number theorem and an inequality of Ramanujan.
\end{abstract}

\maketitle

\section{Introduction}
\subsection{Background}
Let
\begin{equation*}
    \psi(x)=\sum_{n\leq x}\Lambda(n)=\sum_{p^k\leq x}\log p
\end{equation*}
be the Chebyshev prime-counting function over prime powers $p^k$, with integers $k\geq 1$. The (truncated) Riemann--von Mangoldt formula is written as
\begin{equation}\label{riemannvoneq}
    \psi(x)=x-\sum_{\substack{\rho=\beta+i\gamma\\|\gamma|\leq T}}\frac{x^\rho}{\rho}+E(x,T),
\end{equation}
where the sum is over all non-trivial zeros $\rho=\beta+i\gamma$ of the Riemann-zeta function $\zeta(s)$ that have $|\gamma|\leq T$, and $E(x,T)$ is an error term. 

Several authors have provided estimates for $E(x,T)$. With reasonable conditions on $x$ and $T$, the standard bound given in the literature (e.g.\ \cite[Chapter 17]{Davenport_1980}) is
\begin{equation}\label{davenbound}
    E(x,T)=O\left(\frac{x\log^2x}{T}\right),
\end{equation}
which in 1983 was improved by Goldston \cite{Goldston_83} to 
\begin{equation}\label{goldstoneq}
    E(x,T)=O\left(\frac{x\log x\log\log x}{T}\right).
\end{equation}
Under the Riemann hypothesis, one can do even better. Namely, Littlewood \cite{little1924two} proves that with the Riemann hypothesis one gets
\begin{equation}\label{rambound}
    E(x,T)=O\left(\frac{x\log x}{T}\right),
\end{equation}
which was later improved by Goldston \cite{goldston1982result} to $O(x/T)$. Wolke \cite{Wolke_1983} and Ramar\'e \cite{Ramare_16_Perron} also claimed to prove averaged versions of the Riemann-von Mangoldt formula with unconditional $O(x/T)$ error terms. However, through correspondence with Ramar\'e, several errors in these works have been uncovered.

In 2016, Dudek \cite[Theorem 1.3]{Dudek_16p} gave an explicit version of \eqref{davenbound}. Namely, he proved the following theorem.
\begin{theorem}[Dudek]\label{dudekthm}
    Take $50<T<x$ for half an odd integer $x>e^{60}$. Then
    \begin{equation*}
        E(x,T)=O^*\left(\frac{2x\log^2 x}{T}\right).
    \end{equation*}
\end{theorem}
The first author \cite[Theorem 2]{cully2023primes} recently improved on Dudek's result by making \eqref{goldstoneq} explicit. In this paper, we give an explicit $O(x\log x/T)$ bound for $E(x,T)$, thereby making \eqref{rambound} unconditional and explicit. This is done by reworking the papers of Wolke \cite{Wolke_1983} and Ramar{\'e} \cite{Ramare_16_Perron}: we optimise parts of their proofs and avoid the errors related to their averaging arguments. This approach gives significantly better explicit bounds for $E(x,T)$ than in previous works.

Our overarching approach is to split $E(x,T)$ into two smaller error terms, say $E_1(x,T)$ and $E_2(x,T)$. We rework \cite[Theorem 1.1]{Ramare_16_Perron} to bound $E_1(x,T)$ in a general way that can be applied to other arithmetic functions besides $\psi(x)$. Then, we rework the proof of \cite[Theorem 2]{Wolke_1983} and use explicit zero-free regions for $\zeta(s)$ to reach an explicit estimate of the form $E_2(x,T) = O(x/T)$. Such an estimate for $E_2(x,T)$ becomes insignificant compared to that for $E_1(x,T)$ for large $x$.

To demonstrate the usefulness of our results, we use our bounds for $E(x,T)$ to improve the main theorems in \cite{cully2023primes} and \cite{P_T_2021}.

\subsection{Statement of main results}
Our main result is as follows.
\begin{theorem}\label{mainthm}
    For any $\alpha\in(0,1/2]$ there exist constants $M$ and $x_M$ such that for $\max\{51,\log x\}<T<(x^{\alpha}-2)/2$,
    \begin{equation}\label{oureq}
        \psi(x)=x-\sum_{\substack{|\gamma|\leq T}}\frac{x^\rho}{\rho}+O^*\left(M\frac{x\log x}{T}\right)
    \end{equation}
    for all $x\geq x_M$. Some admissible values of $x_M$, $\alpha$ and $M$ are $(40,1/2,5.03)$ and $(10^3, 1/100, 0.5597)$, with more given in Table \ref{maintable} in the appendix.
\end{theorem}
Using Theorem~\ref{mainthm} we are able to obtain the following result.
\begin{theorem}\label{powerthm}
    There is at least one prime between $n^{140}$ and $(n+1)^{140}$ for all $n\geq 1$.
\end{theorem}
This improves upon \cite[Thm.~1]{cully2023primes} by the first author, which asserts that there is always a prime between consecutive $155^{\text{th}}$ powers.

Theorem~\ref{mainthm}, combined with other recent results, also allows us to improve the error term in the prime number theorem for large $x$ (cf. \cite[Theorem 1]{P_T_2021}).
\begin{theorem}\label{errorthm}
    Let $R=5.5666305$. For each $\{X,A,B,C,\epsilon_0\}$ in Table \ref{errortable} we have
    \begin{equation*}
        \left|\frac{\psi(x)-x}{x}\right|\leq A\left(\frac{\log x}{R}\right)^B\exp\left(-C\sqrt{\frac{\log x}{R}}\right),
    \end{equation*}
    and for all $\log x\geq X$,
    \begin{equation*}
        |\psi(x)-x|\leq\epsilon_0x.
    \end{equation*}
\end{theorem}

\def\arraystretch{1.5}
\begin{table}[h]
\centering
\caption{Values of $X$, $A$, $B$, $C$ and $\epsilon_0$ for Theorem~\ref{errorthm}. Here, $\sigma$ is a parameter that appears in the proof of Theorem~\ref{errorthm}. The entry for $X=3600$ is specifically included for Corollary \ref{ramancor}.}
\begin{tabular}{|c|c|c|c|c|c|}
\hline
$X$ & $\sigma$ & $A$ & $B$ & $C$ & $\epsilon_0$\\
\hline
$1000$ & $0.980$ & $269.1$ & $1.520$ & $1.893$ & $6.89\cdot 10^{-6}$ \\
\hline
$2000$ & $0.984$ & $264.8$ & $1.516$ & $1.914$ & $3.48\cdot 10^{-10}$\\
\hline
$3000$ & $0.986$ & $264.3$ & $1.514$ & $1.925$ & $1.42\cdot 10^{-13}$\\
\hline
$3600$ & $0.988$ & $275.2$ & $1.512$ & $1.936$ & $2.04\cdot 10^{-15}$\\
\hline
$4000$ & $0.988$ & $266.5$ & $1.5212$ & $1.936$ & $1.61\cdot 10^{-16}$\\
\hline
$5000$ & $0.990$ & $350.4$ & $1.510$ & $1.946$ & $4.74\cdot 10^{-19}$\\
\hline
$6000$ & $0.990$ & $267.8$ & $1.510$ & $1.946$ & $1.83\cdot 10^{-21}$\\
\hline
$7000$ & $0.990$ & $266.9$ & $1.510$ & $1.946$ & $1.38\cdot 10^{-23}$\\
\hline
$8000$ & $0.990$ & $266.9$ & $1.510$ & $1.946$ & $1.44\cdot 10^{-25}$\\
\hline
$9000$ & $0.992$ & $280.5$ & $1.508$ & $1.957$ & $1.30\cdot 10^{-27}$\\
\hline
$10000$ & $0.992$ & $268.6$ & $1.508$ & $1.957$ & $2.06\cdot 10^{-29}$\\
\hline
\end{tabular}
\label{errortable}
\end{table}

The values of $\epsilon_0$ in Table \ref{errortable} are 40--80$\%$ smaller than those in \cite[Table 1]{P_T_2021}. Replacing $A$ with $A_1=A+0.1$ gives a corresponding expression for $\theta(x)$ (see Corollary \ref{errorcor}). It should be noted that the methods used to prove Theorem~\ref{errorthm} have been expanded on in recent preprints \cite{fiorisharper2022,johnstonsome2022}. 

By repeating the argument on pages 877--880 of \cite{P_T_2021}, we obtain the following application to an inequality of Ramanujan on the prime counting function $\pi(x)$. 
\begin{corollary}\label{ramancor}
    For all $x\geq\exp(3604)$ we have
    \begin{equation*}
        \pi(x)^2<\frac{ex}{\log x}\pi\left(\frac{x}{e}\right).
    \end{equation*}
\end{corollary}
Corollary \ref{ramancor} improves \cite[Theorem 2]{P_T_2021} by a factor of $\exp(311)$.

\subsection{Outline of paper}
We begin in Section \ref{sectram} by making an explicit and simplified version of a truncated Perron formula due to Ramar\'e \cite{Ramare_16_Perron}. In Section \ref{zfsect} we state a number of zero-free regions from the literature that will be required throughout the paper. In Section \ref{wolkesect}, we make explicit an argument of Wolke \cite{Wolke_1983}. These results are combined in Section \ref{mainsect} to prove Theorem~\ref{mainthm}. Applications of these results (Theorems \ref{powerthm} and \ref{errorthm}) are in Sections \ref{powersect} and \ref{errorsect}.

\section{An explicit truncated Perron formula}\label{sectram}
In this section we prove the following error estimate for the truncated Perron formula. This result is a self-contained variant of \cite[Theorem 1.1]{Ramare_16_Perron}, which we have simplified and optimised for our purposes.

\begin{theorem}\label{thm:new-main}
    Let $F(s) = \sum_{n\geq 1} a_n/n^s$ be a Dirichlet series over complex $s$, and $\kappa>0$ be a real parameter chosen larger than the abscissa of absolute convergence of $F(s)$. Also let $\theta'=2/(\sqrt{\pi^2+4}+\pi)$. For any $T>0$, $x\geq 1$, $\kappa>\kappa_a$, and $\lambda\geq\theta'/T$,
    \begin{align*}
        &\sum_{n\leq x} a_n = \frac{1}{2\pi i} \int_{\kappa-iT}^{\kappa+iT} F(s) \frac{x^s}{s}\mathrm{d}s \\
        &\qquad\qquad\qquad\qquad + O^{*}\left(\frac{x^\kappa}{\pi \lambda T} \sum_{n\geq 1} \frac{|a_n|}{n^\kappa}  +  \frac{1}{\pi T}\int_{\theta'/T}^{\lambda} \sum_{|\log(x/n)| \leq u} |a_n| \frac{e^{\kappa u}}{u^2}\mathrm{d}u\right).
    \end{align*}
\end{theorem}

Theorem~\ref{thm:new-main} is proven using Lemma \ref{lem:new}, which is a specific case of \cite[Lem.~2.2]{Ramare_16_Perron}. For both proofs, we will use the step function $$v(y) = \begin{cases} 1 \quad y\geq 0 \\ 0 \quad y<0.\end{cases}$$ 

\begin{lemma}\label{lem:new}
For $\kappa'>0$ and $y\in \mathbb{R}$ we have
\begin{equation}\label{lem}
    \left| v(y) - \frac{1}{2\pi i} \int_{\kappa'-i}^{\kappa'+i} \frac{e^{ys}}{s} \mathrm{d}s \right|  \leq \min\left\lbrace \frac{e^{y\kappa'}}{\pi |y|}, \left| v(y) -\frac{e^{y\kappa'}}{\pi} \arctan(1/\kappa') \right|+\frac{|y|e^{y\kappa'}}{\pi} \right\rbrace.
\end{equation}
\end{lemma}

\begin{proof}
For $y\in\mathbb{R}$ and $K>\kappa'$ we can evaluate the contour integral
\begin{equation*}
    \left( \int_{\kappa'-i}^{\kappa'+i} + \int_{\kappa'+i}^{K+i} + \int_{K+i}^{K-i} + \int_{K-i}^{\kappa'-i} \right) \frac{e^{ys}}{s}\mathrm{d}s = 0.
\end{equation*}
Consider the case $y<0$: as $K$ approaches infinity the third integral approaches zero, and the two horizontal integrals are bounded by $e^{y\kappa'}/|y|$. Hence,
\begin{equation*}
    \left| \int_{\kappa'-i}^{\kappa'+i} \frac{e^{ys}}{s}\mathrm{d}s \right|\leq \frac{2e^{y\kappa'}}{|y|}. 
\end{equation*}
Therefore, we have
\begin{align*}
    \left| v(y) - \frac{1}{2\pi i} \int_{\kappa'-i}^{\kappa'+i} \frac{e^{ys}}{s}\mathrm{d}s \right| &\leq \frac{e^{y\kappa'}}{\pi|y|}.
\end{align*}
The case $y>0$ is similar, but we use a contour extended to the left, so for $K<0$,
\begin{equation*}
    \left( \int_{\kappa'-i}^{\kappa'+i} + \int_{\kappa'+i}^{K+i} + \int_{K+i}^{K-i} + \int_{K-i}^{\kappa'-i} \right) \frac{e^{ys}}{s}\mathrm{d}s = 2\pi i.
\end{equation*}
Taking $K\rightarrow -\infty$ brings the third integral to zero, so
\begin{equation*}
    \left| 2\pi i - \int_{\kappa'-i}^{\kappa'+i} \frac{e^{ys}}{s}\mathrm{d}s \right| \leq \frac{2e^{y\kappa'}}{y},
\end{equation*}
and thus, for $y>0$
\begin{align*}
    \left| v(y) - \frac{1}{2\pi i} \int_{\kappa'-i}^{\kappa'+i} \frac{e^{ys}}{s}\mathrm{d}s \right| &\leq \frac{e^{y\kappa'}}{\pi y}.
\end{align*}
The above bounds are most useful for large $y$. For small $y$ we can use
\begin{equation*}
    \int_{\kappa'-i}^{\kappa'+i} \frac{e^{ys}}{s}\mathrm{d}s = e^{y\kappa'} \int_{\kappa'-i}^{\kappa'+i} \frac{\mathrm{d}s}{s} + i e^{y\kappa'}\int_{-1}^1 \frac{e^{iyt}-1}{\kappa'+it}\mathrm{d}t.
\end{equation*}
The first integral is equivalent to $2i\text{arctan}(1/\kappa')$. The second integral can be bounded using the identity
\begin{equation*}
    \left| \frac{e^{iyt}-1}{iyt} \right| = \left| \int_{0}^1 e^{iytu} \mathrm{d}u \right| \leq 1.
\end{equation*}
Hence, for all $y\in\mathbb{R}$,
\begin{align*}
    \left| v(y) - \frac{1}{2\pi i} \int_{\kappa'-i}^{\kappa'+i} \frac{e^{ys}}{s}\mathrm{d}s \right| &\leq \left| v(y) - \frac{e^{y\kappa'}}{\pi} \text{arctan}(1/\kappa') \right| + \frac{|y|e^{y\kappa'}}{\pi}.\qedhere
\end{align*}
\end{proof}

\begin{proof}[Proof of Theorem~\ref{thm:new-main}]
We aim to bound
\begin{equation*}
    \left| \sum_{n\geq 1} a_n v(T\log(x/n)) - \frac{1}{2\pi i} \int_{\kappa-iT}^{\kappa+iT} F(s) \frac{x^s}{s} \mathrm{d}s \right| x^{-\kappa}
\end{equation*}
for any $x\geq 1$, $T>0$, and $\kappa>\kappa_a>0$, where $\kappa_a$ is the abscissa of absolute convergence of $F(s)$. First, we take $\kappa = \kappa' T$ to reach
\begin{align}
    &\left| \sum_{n\geq 1} a_n v(T\log(x/n)) - \frac{1}{2\pi i} \int_{\kappa-iT}^{\kappa+iT} F(s) \frac{x^s}{s} \mathrm{d}s \right| x^{-\kappa}\notag\\
    &\qquad\leq \sum_{n\geq 1} \frac{|a_n|}{n^\kappa} \left| v(T\log(x/n)) - \frac{1}{2\pi i} \int_{\kappa-iT}^{\kappa+iT} \left(\frac{x}{n}\right)^{s}\frac{\mathrm{d}s}{s} \right| \left(\frac{n}{x}\right)^{\kappa}  \notag\\
    &\qquad= \sum_{n\geq 1} \frac{|a_n|}{n^\kappa} \left| v(T\log(x/n)) - \frac{1}{2\pi i} \int_{\kappa'-i}^{\kappa'+i} e^{T\log(x/n)w}\frac{\mathrm{d}w}{w}  \right| e^{-\kappa' T\log(x/n)}.\label{perroneq1}
\end{align}
Next, we apply Lemma \ref{lem:new}. Suppose for some $\theta> 0$ there exists a positive constant $c$ such that
\begin{align}
    \max_{|y|< \theta} \left( \min\left( \frac{1}{\pi |y|}, \left| \frac{v(y)}{e^{y\kappa'}} - \frac{1}{\pi} \arctan(1/\kappa') \right| + \frac{|y|}{\pi} \right) \right) \leq c  \label{c2}.
\end{align}
Splitting the sum in \eqref{perroneq1} at $\theta$ and taking $y=T\log(x/n)$ in Lemma \ref{lem:new} gives
\begin{align}
    &\left| \sum_{n\geq 1} a_n v(T\log(x/n)) - \frac{1}{2\pi i} \int_{\kappa-iT}^{\kappa+iT} F(s) \frac{x^s}{s} \mathrm{d}s \right| x^{-\kappa}\notag \\
    &\qquad\leq c\sum_{T|\log(x/n)| < \theta} \frac{|a_n|}{n^\kappa} + \frac{1}{\pi T} \sum_{T|\log(x/n)| \geq \theta} \frac{|a_n|}{n^\kappa |\log(x/n)|},\label{perron2eq}
\end{align}
and note that the first bound in (\ref{lem}) was used to obtain the second term in (\ref{perron2eq}). We then have
\begin{align*}
    \sum_{T|\log(x/n)| \geq \theta} \frac{|a_n|}{n^\kappa |\log(x/n)|} &= \sum_{T|\log(x/n)| \geq \theta} \frac{|a_n|}{n^\kappa} \int_{|\log(x/n)|}^\infty \frac{\mathrm{d}u}{u^2} \\
    &= \int_{\theta/T}^\infty \sum_{\theta/T \leq |\log(x/n)| \leq u} \frac{|a_n|}{n^\kappa} \frac{\mathrm{d}u}{u^2} \\
    &= \int_{\theta/T}^\infty \sum_{|\log(x/n)| \leq u} \frac{|a_n|}{n^\kappa} \frac{\mathrm{d}u}{u^2} - \frac{T}{\theta} \sum_{T|\log(x/n)|< \theta} \frac{|a_n|}{n^\kappa}.
\end{align*}
Hence,
\begin{align}
    &\left| \sum_{n\geq 1} a_n v(T\log(x/n)) - \frac{1}{2\pi i} \int_{\kappa-iT}^{\kappa+iT} F(s) \frac{x^s}{s} \mathrm{d}s \right| x^{-\kappa} \notag\\
    &\qquad\leq \frac{1}{\pi T} \int_{\theta/T}^\infty \sum_{|\log(x/n)| \leq u} \frac{|a_n|}{n^\kappa} \frac{\mathrm{d}u}{u^2}  +  \left(c - \frac{1}{\pi \theta} \right) \sum_{T|\log(x/n)|< \theta} \frac{|a_n|}{n^\kappa}. \label{boundWithc2}
\end{align}
Over $T|\log(x/n)|<\theta$ we have $\left| v(y)e^{-y\kappa'} - \arctan(1/\kappa')/\pi \right|<1$, so we can take
\begin{equation*}
    c = \max_{|y|<\theta} \left( \min\left( \frac{1}{\pi |y|}, 1 + \frac{|y|}{\pi} \right) \right) = \frac{\sqrt{\pi^2 + 4}}{2 \pi} + \frac{1}{2}.
\end{equation*}
This implies that if we choose $\theta\leq \theta':=2/(\sqrt{\pi^2+4}+\pi)$ then the second term of (\ref{boundWithc2}) vanishes. Also, the integral in the first term of (\ref{boundWithc2}) will be minimised for larger $\theta$. Hence, we will take $\theta = \theta'$. This leaves us with
\begin{align*}
    &\left| \sum_{n\geq 1} a_n v(T\log(x/n)) - \frac{1}{2\pi i} \int_{\kappa-iT}^{\kappa+iT} F(s) \frac{x^s}{s} \mathrm{d}s \right| \leq \frac{x^{\kappa}}{\pi T} \int_{\theta'/T}^\infty \sum_{|\log(x/n)| \leq u} \frac{|a_n|}{n^\kappa} \frac{\mathrm{d}u}{u^2}.
\end{align*}

Splitting the integral from $\theta'/T$ to $\infty$ at $\lambda$ then gives the further bound
\begin{align*}
    x^\kappa \int_{\theta'/T}^\infty \sum_{|\log(x/n)| \leq u} \frac{|a_n|}{n^\kappa} \frac{\mathrm{d}u}{u^2} &\leq x^\kappa\sum_{n\geq 1} \frac{|a_n|}{n^\kappa} \int_{\lambda}^\infty \frac{\mathrm{d}u}{u^2}  +  \int_{\theta'/T}^{\lambda} \sum_{|\log(x/n)| \leq u} |a_n| \left(\frac{x}{n}\right)^\kappa \frac{\mathrm{d}u}{u^2} \\ 
    &\leq \frac{x^\kappa}{\lambda} \sum_{n\geq 1} \frac{|a_n|}{n^\kappa}  +  \int_{\theta'/T}^{\lambda} \sum_{|\log(x/n)| \leq u} |a_n| \frac{e^{\kappa u}}{u^2} \mathrm{d}u,
\end{align*}
from which Theorem~\ref{thm:new-main} follows.
\end{proof}

\section{Zero-free regions}\label{zfsect}
There are a range of explicit zero-free regions for $\zeta(s)$ in the literature. We will combine several existing results to have an optimal zero-free region for different heights in the complex plane. There is no need for zero-free regions for small $\Im(s)$, however, as the Riemann hypothesis has been verified up to a height $H$. The most recent computation of $H$ is from Platt and Trudgian \cite{P_T-RH_21}.
\begin{lemma}[{\cite{P_T-RH_21}}]\label{rheightlem}
    If $\zeta(\beta+i t)=0$ for any $|t|\leq 3\cdot 10^{12}$ then $\beta=\frac{1}{2}$.
\end{lemma}

The most recent explicit version of the ``classical" zero-free region is from Mossinghoff and Trudgian \cite{M_T_2015}.

\begin{lemma}[{\cite{M_T_2015}}]\label{classlem}
    For $|t|\geq 2$ there are no zeros of $\zeta(\beta+it)$ in the region $\beta\geq 1-\nu_1(t)$, where
    \begin{equation*}
        \nu_1(t)=\frac{1}{R_0\log |t|}
    \end{equation*}
    and\footnote{This value of $R_0$ is lower than that appearing in \cite[Theorem 1]{M_T_2015}. However, since the Riemann hypothesis has now been verified to a higher height (Lemma~\ref{rheightlem}), we can take $R_0=5.5666305$ as discussed in \cite[Section 6.1]{M_T_2015}.} $R_0=5.5666305$.
\end{lemma}

Ford \cite{Ford_2002} gave a version of the classical result which improves on Mossinghoff and Trudgian's result for large $|t|$.

\begin{lemma}\label{fordclassicregion}
    For $|t|\geq 5.45\cdot 10^8$ there are no zeros of $\zeta(\beta+it)$ in the region $\beta\geq 1-\nu_2(t)$, where
    \begin{equation*}
       \nu_2(t)=\frac{1}{R(|t|)\log |t|},
    \end{equation*}
    with
    \begin{equation*}
        R(t)=\frac{J(t)+0.685+0.155\log\log t}{\log t\left(0.04962-\frac{0.0196}{J(t)+1.15}\right)},
    \end{equation*}
    and
    \begin{equation*}
        J(t)=\frac{1}{6}\log t+\log\log t+\log(0.63).
    \end{equation*}
\end{lemma}
\begin{proof}
    This result is almost identical to \cite[Theorem 3]{Ford_2002}, except we use an improved constant in $J(t)$. Ford's (1.6) can be replaced with a more recent result due to Hiary \cite[Theorem 1.1]{hiary2016explicit}. We note however, that an error was recently discovered in \cite{hiary2016explicit} (see \cite[\S 2.2.1]{patel2021explicit}). Despite this, Hiary's result can be recovered (and in fact improved) as discussed in the preprint \cite{hiary2022improved}.
\end{proof}

For very large $t$, we use an explicit Vinogradov--Korobov zero-free region, also due to Ford \cite{Ford_2002}.

\begin{lemma}[{\cite[Theorem 5]{Ford_2002}}]\label{vklem}
    For $|t|\geq 3$ there are no zeros of $\zeta(\beta+it)$ in the region $\beta\geq 1-\nu_3(t)$ where
    \begin{equation}\label{vkregion}
        \nu_3(t)=\frac{1}{c\log^{2/3}|t|(\log\log |t|)^{1/3}}
    \end{equation}
    and $c=57.54$.
\end{lemma}

To use the widest zero-free region we set 
\begin{equation}\label{mainnueq}
    \nu(t)=
    \begin{cases}
        \frac{1}{2},&\text{if $|t|\leq 3\cdot 10^{12}$},\\
        \max\{\nu_1(t),\nu_2(t),\nu_3(t)\},&\text{otherwise,}
    \end{cases}    
\end{equation}
noting that $\nu_2(t)\geq\nu_1(t)$ for $t\geq\exp(46.3)$ and $\nu_3(t)\geq\nu_2(t)$ for $t\geq\exp(54599)$.

\section{Wolke's method}\label{wolkesect}
In this section we prove an explicit bound for the integral in Theorem~\ref{thm:new-main},
\begin{equation*}
    \frac{1}{2\pi i}\int_{\kappa-iT}^{\kappa+iT}F(s)\frac{x^s}{s}\mathrm{d}s,
\end{equation*}
for the case $\kappa=1+1/\log x$ and $F(s)=\sum_{n\geq 1}\Lambda(n) n^{-s} = -(\zeta'/\zeta)(s)$.
\begin{theorem}\label{wolkeprop}
    Let $\alpha\in(0,1]$ and $\omega\in[0,1]$. There exists constants $K$ and $x_K$ such that if $x\geq x_K$ and $\max\{51,\log x\}<T<\frac{x^{\alpha}}{2}-1$, 
    \begin{equation}\label{wolkeeq}
        \frac{1}{2\pi i}\int_{1+\varepsilon-iT}^{1+\varepsilon+iT}\left(-\frac{\zeta'}{\zeta}(s)\right) \frac{x^s}{s}\mathrm{d}s = x-\sum_{\substack{\rho=\beta+i\gamma\\|\gamma|\leq T}}\frac{x^\rho}{\rho}+O^*\left(\frac{Kx}{T}(\log x)^{1-\omega}\right),
    \end{equation}
    where $\varepsilon=1/\log x$. Some corresponding values of $\alpha$, $\omega$, $K$ and $x_K$ are given in Table \ref{theorem2table1} in the appendix.
\end{theorem}
Although setting $\omega=1$ gives the best result asymptotically, the resulting constant will be quite large unless $x_K$ is very large. For the best results we take $\omega$ closer to 1 for large $x$ and small $T$, and $\omega$ closer to 0 for small $x$ and large $T$.

To prove Theorem~\ref{wolkeprop} we make \cite[Theorem 2]{Wolke_1983} explicit. The restriction $T>51$ is so we can directly use some results from \cite{Dudek_16p}, which contains a proof of an expression similar to \eqref{wolkeeq} albeit with a weaker error term. The restriction $T>\log x$ is not necessary but allows us to obtain a better value of $K$ for large $x$.

In what follows, let $\overline{\omega}\geq 0$ be a parameter to be optimised in our computations. The value $\omega$ appearing in Theorem~\ref{wolkeprop} will be given by
\begin{equation*}
    \omega=
    \begin{cases}
        \overline{\omega},&\text{if $\overline{\omega}<1$},\\
        1,&\text{if $\overline{\omega}\geq 1$}.
    \end{cases}
\end{equation*}

We also make use of explicit zero-free regions, and set $\nu(t)$ as in \eqref{mainnueq}.

We now prove a couple of preliminary lemmas. In what follows, $N(T)$ denotes the number of zeros of $\zeta(s)$ with imaginary part up to height $T$.

\begin{lemma}\label{plus1minus1lem}
    Let $t>1$. Then $N(t+1)-N(t-1)<\log t$.
\end{lemma}
\begin{proof}
    Dudek \cite[Lemma 2.6]{Dudek_16p} proves this lemma for $t>50$. For $1<t\leq 50$ we used Odlyzko's tables \cite{odlyzkotable} to verify the lemma manually. 
\end{proof}

\begin{lemma}\label{djslem}
    Let $T>51$ and $x>e$. There exists a $\tau\in(T-1/(\log x)^{\overline{\omega}},T+1/(\log x)^{\overline{\omega}})$ such that when $s=\sigma+i\tau$ with $\sigma>-1$, we have
    \begin{equation*}
        \left|\frac{\zeta'}{\zeta}(s)\right|<(\log x)^{\overline{\omega}}(\log^2T+\log T)+20\log T.
    \end{equation*}
\end{lemma}
\begin{proof}
    By Lemma \ref{plus1minus1lem},
    \begin{equation}\label{ntbounds}
        N(T+1/(\log x)^{\overline{\omega}})-N(T-1/(\log x)^{\overline{\omega}})\leq N(T+1)-N(T-1)<\log T.   
    \end{equation}
    As there are at most $\log T$ zeros of $\zeta(s)$ with imaginary part in the interval $(T-1/(\log x)^{\overline{\omega}},T+1/(\log x)^{\overline{\omega}})$, we can split $(T-1/(\log x)^{\overline{\omega}},T+1/(\log x)^{\overline{\omega}})$ into at most $\lfloor\log T\rfloor+1$ zero-free regions. At least one of these regions will have height greater than 
    \begin{equation*}
        \frac{2}{(\log x)^{\overline{\omega}}(\log T+1)}.
    \end{equation*}
    We define $\tau$ to be the midpoint of such a region so that 
    \begin{equation}\label{distanceeq}
       |\tau-\gamma| \geq\frac{1}{(\log x)^{\overline{\omega}}(\log T+1)} 
    \end{equation}
    for all zeros $\rho=\beta+i\gamma$. Next we use the following result which holds for all $\sigma>-1$ and $\tau>50$ \cite[p. 187]{Dudek_16p}
    \begin{equation}\label{logderbound}
        \frac{\zeta'}{\zeta}(s)=\sum_{|\gamma-\tau|<1}\frac{1}{s-\rho}+O^*(19\log \tau).
    \end{equation}
    Using \eqref{ntbounds} and \eqref{distanceeq},
    \begin{equation*}
        \left|\sum_{|\gamma-\tau|<1}\frac{1}{s-\rho}\right|\leq\sum_{|\gamma-\tau|<1}\frac{1}{|\tau-\gamma|}<(\log x)^{\overline{\omega}}(\log T+1)\log T
    \end{equation*}
    which completes the proof of the lemma upon noting that $19\log\tau\leq 20\log T$.
\end{proof}

We are now in a position to prove Theorem~\ref{wolkeprop}.
\begin{proof}[Proof of Theorem~\ref{wolkeprop}]
Letting $\tau$ be as in Lemma \ref{djslem}, we define the contour $C=C_1\cup C_2\cup C_3\cup C_4$, where
\begin{align*}
    &C_1=[1+\varepsilon-i\tau,1+\varepsilon+i\tau], &C_2=[1+\varepsilon+i\tau,-1+i\tau],\\
    &C_3=[-1+i\tau,-1-i\tau], &C_4=[-1-i\tau,1+\varepsilon-i\tau].
\end{align*}
Cauchy's residue theorem gives
\begin{align}\label{maintaueq}
    \frac{1}{2\pi i}\int_{1+\varepsilon-i\tau}^{1+\varepsilon+i\tau}&\left(-\frac{\zeta'}{\zeta}(s)\right)\frac{x^s}{s}\mathrm{d}s\notag\\
    &=x-\sum_{|\gamma|\leq\tau}\frac{x^\rho}{\rho}-\log(2\pi)-\frac{1}{2\pi i}\int_{C_2\cup C_3\cup C_4}\left(-\frac{\zeta'}{\zeta}(s)\right)\frac{x^s}{s}\mathrm{d}s,
\end{align}
noting that $(\zeta'/\zeta)(0)=\log(2\pi)$. We begin by converting the left-hand side of \eqref{maintaueq} into an integral involving $T$ as opposed to $\tau$. To do this, we note that when $\Re(s)=1+\varepsilon$, the main theorem in \cite{Delange_87} gives
\begin{equation*}
    \left|\frac{\zeta'}{\zeta}(s)\right|\leq\left|\frac{\zeta'}{\zeta}(1+\varepsilon)\right|<\log x.
\end{equation*}
Hence,
\begin{align}\label{deltaterm}
    \left|\frac{1}{2\pi i}\int_{1+\varepsilon+iT}^{1+\varepsilon+i\tau}\left(-\frac{\zeta'}{\zeta}(s)\right)\frac{x^s}{s}\mathrm{d}s\right|&\leq\frac{ex}{2\pi (T-1)}|\tau-T|\log x \notag\\
    &\leq\frac{ex}{2\pi (T-1)}(\log x)^{1-\overline{\omega}}
\end{align}
and identically for the integral from $1+\varepsilon-i\tau$ to $1+\varepsilon-iT$. Therefore,
\begin{align*}
    \frac{1}{2\pi i}\int_{1+\varepsilon-i\tau}^{1+\varepsilon+i\tau}\left(-\frac{\zeta'}{\zeta}(s)\right)\frac{x^s}{s}\mathrm{d}s=\frac{1}{2\pi i}\int_{1+\varepsilon-iT}^{1+\varepsilon+iT}&\left(-\frac{\zeta'}{\zeta}(s)\right)\frac{x^s}{s}\mathrm{d}s\\
    &+O^*\left(\frac{ex}{\pi(T-1)}(\log x)^{1-\overline{\omega}}\right).
\end{align*}
Next we consider the sum on the right-hand side of \eqref{maintaueq}. In the case $T<\tau$,
\begin{equation*}
    \sum_{|\gamma|\leq\tau}\frac{x^{\rho}}{\rho}=\sum_{|\gamma|\leq  T}\frac{x^{\rho}}{\rho}+\sum_{T<|\gamma|\leq \tau}\frac{x^{\rho}}{\rho}.
\end{equation*}
By \eqref{ntbounds} we then have
\begin{equation*}
    \left|\sum_{T<|\gamma|\leq \tau}\frac{x^{\rho}}{\rho}\right|\leq 2\log \tau\frac{x^{1-\nu(\tau)}}{T}\leq \frac{2x\log (T+1)}{x^{\nu(T+1)}T}.
\end{equation*} 
The case $T>\tau$ is similar, with an error bounded by $2x\log T/x^{\nu(T)}(T-1)$. Hence, 
\begin{equation}\label{minitoughterm}
    \sum_{|\gamma|\leq\tau}\frac{x^{\rho}}{\rho}=\sum_{|\gamma|\leq  T}\frac{x^{\rho}}{\rho}+O^*\left(\frac{2x\log (T+1)}{x^{\nu(T+1)}(T-1)}\right).
\end{equation}
Finally we deal with the integral on the right-hand side of \eqref{maintaueq}. For the integral over $C_3$ we use the inequality \cite[Lemma 2.3]{Dudek_16p}
\begin{equation*}
    \left|\frac{\zeta'}{\zeta}(s)\right|<9+\log|s|,
\end{equation*}
which holds when $\Re(s)=-1$. In particular, 
\begin{align*}
    \left|\frac{1}{2\pi i}\int_{C_3}\left(-\frac{\zeta'}{
    \zeta}(s)\right)\frac{x^s}{s}\mathrm{d}s\right|&\leq\frac{2\tau}{2\pi }\frac{9+\log(\sqrt{\tau^2+1})}{x}\\
    &\leq\frac{(T+1)(9+\log(\sqrt{(T+1)^2+1})}{\pi x}.
\end{align*}
Next we bound the integral over $C_2$. The same bound will also hold for the integral over $C_4$ by symmetry.\\
\\
\textbf{Case 1: $\omega=0$.}

When $\omega=\overline{\omega}=0$ we bound the integral over $C_2$ using Lemma \ref{djslem}. That is,
\begin{align}\label{omegazero}
    \left|\frac{1}{2\pi i}\int_{C_2}\left(-\frac{\zeta'}{\zeta}(s)\right)\frac{x^s}{s}\mathrm{d}s\right|&\leq\frac{\log^2T+21\log T}{2\pi\tau}\int_{-1}^{1+\varepsilon}x^\sigma\mathrm{d}\sigma\notag\notag\\
    &\leq\frac{ex(\log^2T+21\log T)}{2\pi\log x(T-1)}.
\end{align}
\ \\
\textbf{Case 2: $\omega>0$.}

In the case where $\omega>0$ we need to be more careful since the expression in \eqref{omegazero} is not $O(x/T)$ (unless further restrictions are placed on $T$). To overcome this issue, we let $D\in(0,1)$ be a parameter that we will optimise for different entries in Table \ref{theorem2table1} and let $\sigma_0=\sigma_0(x,\alpha):=1-D\nu(x^{\alpha})$. We then split $C_2$ into $C_2=C_{21}\cup C_{22}$ where
\begin{equation*}
    C_{21}=[1+\varepsilon+i\tau,\sigma_0+i\tau],\quad C_{22}=[\sigma_0+i\tau,-1+i\tau].
\end{equation*}
By Lemma \ref{djslem}
\begin{align}\label{toughterm}
    \left|\frac{1}{2\pi i}\int_{C_{22}}\left(-\frac{\zeta'}{\zeta}(s)\right)\frac{x^s}{s}\mathrm{d}s\right|&\leq\frac{(\log x)^{\overline{\omega}}(\log^2T+\log T)+20\log T}{2\pi\tau}\int_{-1}^{\sigma_0}x^\sigma\mathrm{d}\sigma\notag\\
    &\leq x\frac{(\log x)^{\overline{\omega}-1}(\log^2T+\log T)+20\log T/\log x}{2\pi x^{D \nu(x^{\alpha})}(T-1)}.
\end{align}
For the integral over $C_{21}$ we follow Wolke and use the following formula for $\zeta'/\zeta$ \cite[Lemma 2]{selberg1943normal}. Let, for $y=x^{1/4}>1$,
\begin{equation*}
    \Lambda_y(n)=
    \begin{cases}
        \Lambda(n),&1\leq n\leq y\\
        \Lambda(n)\frac{\log(y^2/n)}{\log y},&y\leq n\leq y^2. 
    \end{cases}
\end{equation*}
Then if $s\neq 1$ and $s$ is not a zero of $\zeta$,
\begin{align*}
    -\frac{\zeta'}{\zeta}(s)&=\sum_{n\leq y^2}\frac{\Lambda_y(n)}{n^s}+\frac{y^{1-s}-y^{2(1-s)}}{(1-s)^2\log y}-\frac{1}{\log y}\sum_{q=1}^\infty\frac{y^{-2q-s}-y^{-2(2q+s)}}{(2q+s)^2}\\
    &\qquad\qquad\qquad\qquad\qquad\qquad\qquad-\frac{1}{\log y}\sum_{\rho}\frac{y^{\rho-s}-y^{2(\rho-s)}}{(s-\rho)^2}\\
    &=Z_1(s)+Z_2(s)+Z_3(s)+Z_4(s),\quad\text{say}.
\end{align*}
For $Z_1$,
\begin{align}\label{z1boundeq}
    \left|\frac{1}{2\pi i}\int_{C_{21}}Z_1(s)\frac{x^s}{s}\mathrm{d}s\right|&\leq\frac{1}{2\pi\tau}\sum_{n\leq y^2}\Lambda_y(n)\int_{\sigma_0}^{1+\varepsilon}\left(\frac{x}{n}\right)^\sigma\mathrm{d}\sigma\notag\\
    &\leq\frac{1}{2\pi(T-1)}\sum_{n\leq y^2}\Lambda_y(n)\frac{(x/n)^{1+\varepsilon}}{\log(x/n)}\notag\\
    &\leq\frac{ex}{2\pi(T-1)}\sum_{n\leq x^{1/2}}\frac{\Lambda(n)}{n^c\log(x/n)}\notag\\
    &\leq\frac{ex}{\pi\log x(T-1)}\sum_{n\leq x^{1/2}}\frac{\Lambda(n)}{n^c}.
\end{align}
Then by the corollary of the main theorem in \cite{Ramare_2013},
\begin{equation}\label{ramarelambdaeq}
    \sum_{n\leq x^{1/2}}\frac{\Lambda(n)}{n}\leq \frac{1}{2}\log x-\gamma+\frac{8}{\log^2x}\leq\frac{1}{2}\log x,
\end{equation}
for $x>50$. By partial summation, we then deduce from \eqref{ramarelambdaeq} that
\begin{equation*}
    \sum_{n\geq x^{1/2}}\frac{\Lambda(n)}{n^c}\leq\frac{1}{e}\log x.
\end{equation*}
This gives an $O(x/T)$ bound for \eqref{z1boundeq}.

For $Z_2$,
\begin{align*}
    \left|\frac{1}{2\pi i}\int_{C_{21}}Z_2(s)\frac{x^s}{s}\mathrm{d}s\right|&\leq\frac{1}{2\pi\tau^3\log y}\int_{\sigma_0}^{1+\varepsilon}(x^{\frac{1-\sigma}{4}}+x^{\frac{1-\sigma}{2}})x^\sigma\mathrm{d}\sigma\\
    &\leq\frac{2}{\pi(T-1)^3\log x}\left(\frac{4}{3}\frac{x^{\frac{1+3(1+\varepsilon)}{4}}}{\log x}+2\frac{x^{\frac{1+(1+\varepsilon)}{2}}}{\log x}\right)\\
    &\leq\frac{4x\left(\frac{2}{3}e^{3/4}+e^{1/2}\right)}{\pi(T-1)^3\log^2x}.
\end{align*}

For $Z_3$,
\begin{align*}
    \left|\frac{1}{2\pi i}\int_{C_{21}}Z_3(s)\frac{x^s}{s}\mathrm{d}s\right|&\leq\frac{1}{2\pi\tau^3\log y}\int_{\sigma_0}^{1+\varepsilon}\left(y^{-\sigma}\sum_{q=1}^\infty y^{-2q}+y^{-2\sigma}\sum_{q=1}^\infty y^{-4q}\right)x^{\sigma}\mathrm{d}\sigma\\
    &\leq\frac{2}{\pi(T-1)^3\log x}\left(\frac{(ex)^{3/4}}{\log (x^{3/4})}\frac{1}{y^2-1}+\frac{(ex)^{1/2}}{\log(x^{1/2})}\frac{1}{y^4-1}\right)\\
    &=\frac{2x}{\pi (T-1)^3\log ^2x}\left(\frac{4}{3}\frac{e^{3/4}}{x^{3/4}-x^{1/4}}+\frac{2e^{1/2}}{x^{3/2}-x^{1/2}}\right).
\end{align*}
For $Z_4$, we first note that
\begin{equation}\label{terminaleq}
    1+\varepsilon-\sigma_0=\frac{1}{\log x}+D\nu(x^{\alpha})=d_{\alpha}(x)\nu(x^{\alpha}),
\end{equation}
where $d_{\alpha}(x):=D+1/\nu(x^{\alpha})\log x$ is non-increasing. Moreover, for a zero $\rho=\beta+i\gamma$ with $|\gamma|\leq 2T+2<x^\alpha$,
\begin{equation}\label{betasigmaeq}
    \sigma_0-\beta>(1-D)\nu(x^{\alpha})=:\nu_{\alpha}(x).
\end{equation}
Now,
\begin{align}
    &\left|\frac{1}{2\pi i}\int_{C_{21}}Z_4(s)\frac{x^s}{s} \mathrm{d}s\right|\\
    &\qquad\leq\frac{1}{2\pi\log y}\sum_{\substack{\rho\\|\gamma|\leq2T+2}}\frac{ex}{\tau}\int_{\sigma_0}^{1+\varepsilon}\frac{y^{\beta-\sigma}+y^{2(\beta-\sigma)}}{|\sigma+i\tau-\rho|^2}\mathrm{d}\sigma\notag\\
    &\qquad\qquad+\frac{1}{2\pi\log y}\sum_{\substack{\rho\\|\gamma|>2T+2}}\frac{x}{\tau}\int_{\sigma_0}^{1+\varepsilon}\frac{\left(\frac{x}{y}\right)^{\sigma-1}y^{\beta-1}+\left(\frac{x}{y^2}\right)^{\sigma-1}y^{2(\beta-1)}}{|\sigma+i\tau-\rho|^2}\mathrm{d}\sigma\notag\\
    &\qquad=\frac{2x}{\pi\tau\log x}\left(eF_1(x,\tau)+F_2(x,\tau)\right),\quad\text{say.}\label{z4inteq}
\end{align}
To bound $F_1(x,\tau)$ we will use the estimate
\begin{equation}\label{f1est1}
    \int_{\sigma_0}^{1+\varepsilon}\frac{\mathrm{d}\sigma}{|\sigma+i\tau-\rho|^2}\leq\frac{1}{|\tau-\gamma|^2}d_{\alpha}(x)v(x^\alpha). 
\end{equation}
whenever $|\tau-\gamma|>2$. On the other hand, when $|\tau-\gamma|\leq 2$, we use the estimate
\begin{align}\label{f1est2}
    \int_{\sigma_0}^{1+\varepsilon}\frac{\mathrm{d}\sigma}{|\sigma+i\tau-\rho|^2}\leq\int_{\sigma_0}^{1+\varepsilon}\frac{\mathrm{d}\sigma}{(\sigma-\beta)^2}&=\frac{1}{\sigma_0-\beta}-\frac{1}{1+\varepsilon-\beta}\notag\\
    &<\frac{d_{\alpha}(x)}{(d_{\alpha}(x)+1-D)v_{\alpha}(x)}.  
\end{align}
Using \eqref{f1est1}, \eqref{f1est2} and Lemma \ref{plus1minus1lem} we have that $F_1(x,\tau)$ is bounded above by
\begin{align}\label{toughterm2}
    &\frac{1}{x^{\nu_{\alpha}(x)/4}+x^{\nu_{\alpha}(x)/2}}\left(\sum_{\substack{|\gamma|\leq 2T+2\\|\tau-\gamma|\leq 2}}+\sum_{h=2}^\infty\sum_{\substack{|\gamma|\leq 2T+2\\2(h-1)<|\tau-\gamma|\leq 2h}}\right)\int_{\sigma_0}^{1+\varepsilon}\frac{d\sigma}{|\sigma+i\tau-\rho|^2}\notag\\
    &\leq\frac{1}{x^{\nu_{\alpha}(x)/4}+x^{\nu_{\alpha}(x)/2}}\left(\frac{2d_{\alpha}(x)\log\tau}{(d_{\alpha}(x)+1-D)v_{\alpha}(x)}\right.\notag\\
    &\qquad\qquad\qquad\qquad\qquad\qquad\qquad\left.+2d_{\alpha}(x)\nu(x^{\alpha})\log(2T+1)\sum_{m=1}^\infty\frac{1}{(2m)^2}\right)\notag\\
    &\leq\frac{2d_{\alpha}(x)\log(x^\alpha)}{x^{\nu_{\alpha}(x)/4}+x^{\nu_{\alpha}(x)/2}}\left(\frac{1}{(d_{\alpha}(x)+1-D)v_{\alpha}(x)}+\frac{v(x^\alpha)\pi^2}{24}\right).
\end{align} 

Then,
\begin{align*}
    F_2(x,\tau)&\leq\sum_{\substack{\rho\\|\gamma|>2T+2}}\frac{1}{\left|\frac{\gamma}{2}\right|^2}\int_{\sigma_0}^{1+\varepsilon}(x^{3/4})^{\sigma-1}+(x^{1/2})^{\sigma-1}\mathrm{d}\sigma\\
    &\leq 4\left(\frac{e^{3/4}}{\log(x^{3/4})}+\frac{e^{1/2}}{\log(x^{1/2})}\right)\sum_{\rho}\frac{1}{|\gamma|^2}\\
    &=\frac{4}{\log x}\left(\frac{4}{3}e^{3/4}+2e^{1/2}\right)\sum_{\rho}\frac{1}{|\gamma|^2},
\end{align*}
where $\sum_{\rho} 1/|\gamma|^2\leq 0.04621$ by \cite[Example 1]{brent2021accurate}.

Combining all of our estimates we see that
\begin{equation*}
    \frac{1}{2\pi i}\int_{1+\varepsilon-iT}^{1+\varepsilon+iT}\left(-\frac{\zeta'}{\zeta}(s) \right) \frac{x^s}{s}\mathrm{d}s =x-\sum_{|\gamma|\leq T}\frac{x^\rho}{\rho}+E(x,T),
\end{equation*}
where $|E(x,T)|$ is bounded above by
\begin{align}\label{bigerror1}
    \log(2\pi)+\frac{ex}{\pi(T-1)}\log x+\frac{2x\log (T+1)}{x^{\nu(T+1)}T}&+\frac{(T+1)(9+\log(\sqrt{(T+1)^2+1}))}{\pi x}\notag\\
    &\qquad +\frac{ex(\log^2T+21\log T)}{\pi\log x(T-1)}
\end{align}
when $\omega=\overline{\omega}=0$ and
\begin{align}\label{bigerror2}
    &\log(2\pi)+\frac{ex}{\pi(T-1)}(\log x)^{1-\overline{\omega}}+\frac{2x\log (T+1)}{x^{\nu(T+1)}T}+\frac{(T+1)(9+\log(\sqrt{(T+1)^2+1}))}{\pi x}\notag\\
    &\ +2\left(x\frac{(\log x)^{\overline{\omega}-1}(\log^2T+\log T)+20\log T/\log x}{2\pi x^{D \nu(x^{\alpha})}(T-1)}+\frac{x}{\pi(T-1)}\right.\notag\\
    &\qquad\quad +\frac{4x\left(\frac{2}{3}e^{3/4}+e^{1/2}\right)}{\pi(T-1)^3\log^2x}+\frac{2x}{\pi (T-1)^3\log ^2x}\left(\frac{4}{3}\frac{e^{3/4}}{x^{3/4}-x^{1/4}}+\frac{2e^{1/2}}{x^{3/2}-x^{1/2}}\right)\notag\\
    &\qquad\quad +\frac{4\alpha exd_{\alpha}(x)}{\pi(T-1)(x^{\nu_{\alpha}(x)/4}+x^{\nu_{\alpha}(x)/2})}\left(\frac{1}{(d_{\alpha}(x)+1-D)v_{\alpha}(x)}+\frac{v(x^\alpha)\pi^2}{24}\right)\notag\\
    &\left.\qquad\quad+\frac{8x}{\pi(T-1)\log^2 x}\left(\frac{4}{3}e^{3/4}+2e^{1/2}\right)\sum_{\rho}\frac{1}{|\gamma|^2}\right)
\end{align}
otherwise. Note that the terms multiplied by 2 in \eqref{bigerror2} are those that occur when bounding both the integral over $C_2$ and $C_4$. Thus, we can write $E(x,T)=O^*(Kx(\log x)^{1-\omega}/T)$ where $K$ can be computed by evaluating an upper bound for each error term in \eqref{bigerror1} or \eqref{bigerror2} in the range $x\geq x_K$.
\end{proof}

Extra care must be taken when bounding some of the error terms in the above proof due to $\nu(t)$ being defined as a composite function. Most notably, we need to be careful of the behaviour at the crossover point $\lambda=54598.16\ldots$ such that $\nu_3(t)\geq\nu_2(t)$ for all $t\geq\exp(\lambda)$ or equivalently $\nu_3(t^\alpha)\geq\nu_2(t^\alpha)$ for all $t\geq \exp(\lambda/\alpha)$.

For example, the error in \eqref{toughterm} can be written as $H(x,T)(x(\log x)^{1-\omega}/T)$ where $H$ decreases in $x$ whilst $\nu(T)=1/2$, and increases in $x$ whilst $\nu(T)=\nu_1(T)$ or $\nu(T)=\nu_2(T)$. Thus, in most cases, an upper bound for $H(x,T)$ occurs at or beyond $x=\exp(\lambda/\alpha)$ corresponding to the upper bound $T<\frac{x^\alpha-2}{4}<x^\alpha$ for $T$. As a result, in the cases where $x_K<\exp(\lambda/\alpha)$ we evaluated \eqref{toughterm} at $x=\exp(\lambda/\alpha)$ to obtain an upper bound. In such cases we also had to verify that \eqref{toughterm} was decreasing for $x\geq\exp(\lambda/\alpha)$. A similar treatment is required for the terms in \eqref{minitoughterm} and \eqref{toughterm2}.

We also remark that the error term 
\begin{equation*}
    \frac{2x}{\pi(T-1)}
\end{equation*}
in \eqref{bigerror2} is bounded below by $2x/\pi T$ independent of the choice of $\alpha$ or $\overline{\omega}$. Thus, $K$ is bounded below by $2/\pi=0.63661\ldots$ when $\omega=1$.\\
\\
Since our computations rely heavily on zero-free regions, we can obtain significant improvements by assuming the Riemann hypothesis. In particular, we have the following result.

\begin{theorem}\label{riemannthm1}
    Assuming the Riemann hypothesis, we can take $K=0.6373$ in \eqref{wolkeeq} for $\omega=1$, $\alpha=1/2$ and $\log x\geq 1000$.
\end{theorem}
\begin{proof}
    Redefine $\nu(t)=\frac{1}{2}$ in the proof of Theorem~\ref{wolkeprop}. The parameters we used to obtain $K=0.6373$ were $D=0.4$ and $\overline{\omega}=3$.
\end{proof}
\begin{remark}
    With more work, the value of $K$ in Theorem~\ref{riemannthm1} can be lowered and in fact made arbitrarily small as $x\to\infty$. In particular, one could bound the integral over $C_2$ using conditional estimates on $\zeta'/\zeta$ (e.g.\ \cite[Corollary 1]{simonic2022estimates}).
\end{remark}

\section{Proof of Theorem~\ref{mainthm}}\label{mainsect}
Using the results of Sections \ref{sectram} and \ref{wolkesect}, we now prove Theorem~\ref{mainthm}. To begin with, we let $x\geq x_M\geq e^{40}$,  $\max\{51,\log x\}<T<\frac{x^{\alpha}}{2}-1$, $\kappa=1+1/\log x$ and set $a_n=\Lambda(n)$ in Theorem~\ref{thm:new-main} to obtain
\begin{align}\label{perronbound1}
        &\psi(x) = \frac{1}{2\pi i} \int_{\kappa-iT}^{\kappa+iT} \left(-\frac{\zeta'}{\zeta}(s)\right) \frac{x^s}{s} \mathrm{d}s \notag\\
        &\qquad\qquad\qquad +O^{*}\left(\frac{x^\kappa}{\pi\lambda T} \sum_{n\geq 1} \frac{\Lambda(n)}{n^\kappa}  +  \frac{1}{\pi T}\int_{\theta'/T}^{\lambda} \sum_{|\log(x/n)| \leq u} \Lambda(n) \frac{e^{\kappa u}}{u^2}\mathrm{d}u\right).
\end{align}
For the first term, we have by Theorem~\ref{wolkeprop},
\begin{equation}\label{newwolkebound}
    \frac{1}{2\pi i}\int_{\kappa-iT}^{\kappa+iT}\left(-\frac{\zeta'}{\zeta}(s)\right)\frac{x^s}{s}\mathrm{d}s=x+\sum_{|\gamma|\leq T}\frac{x^{\rho}}{\rho}+O^*\left(K\frac{x}{T}(\log x)^{1-\omega}\right),
\end{equation}
for some constant $K$ corresponding to $x_K=x_M$. Then, the first sum in \eqref{perronbound1} is bounded by
\begin{equation}\label{firstsumbound}
    \frac{x^\kappa}{\lambda\pi T} \sum_{n\geq 1} \frac{\Lambda(n)}{n^\kappa}\leq\frac{ex\log x}{\lambda\pi T}
\end{equation}
by the main theorem in \cite{Delange_87}. For the second term in \eqref{perronbound1}, the condition $|\log(x/n)|\leq u$ is equivalent to $xe^{-u}\leq n\leq xe^u$. For all $u\in[0,\lambda]$ we have $e^{-u}\geq 1-u$ and $e^u\leq 1+(c_0-1)u$ with $$c_0 = c_0(\lambda) = \frac{e^\lambda-1}{\lambda}+1.$$ Hence, the sum simplifies to
\begin{align}\label{perronbound2}
    \frac{1}{\pi T}\int_{\theta'/T}^{\lambda} \sum_{|\log(x/n)| \leq u} \Lambda(n) \frac{e^{\kappa u}}{u^2}\mathrm{d}u&\leq \frac{e^{\kappa\lambda}}{\pi T}\int^{\lambda}_{\theta'/T}\frac{1}{u^2}\sum_{I(x,u)}\Lambda(n)\mathrm{d}u,
\end{align}
where
\begin{equation*}
    I(x,u)=\left\{n\geq 1\::\:x-ux\leq n\leq x + (c_0-1)ux \right\}.
\end{equation*}
Let $x_-=\max\{x-ux,0\}$ and $x_+ = x+(c_0-1)ux$ for $\theta'/T\leq u\leq \lambda$. Defining 
\begin{equation*}
    \theta(x)=\sum_{p\leq x}\log p,
\end{equation*}
we have by an explicit form of the Brun--Titchmarsh theorem \cite[Theorem 2]{montgomery1973large},
\begin{align*}
    \theta(x_+)-\theta(x_-)&\leq \frac{2 \log x_+}{\log(x_+-x_-)}(x_+-x_-) \\
    &\leq \frac{2 \log( x+(c_0-1)ux)}{\log(c_0 ux)}c_0ux \\
    &\leq c_0 ux\cdot\mathcal{E}_1(x,T),
\end{align*}
where
\begin{equation*}
    \mathcal{E}_1(x,T)=\frac{2 \log( x+(c_0-1)\lambda x)}{\log(c_0 \theta'x/T)}.
\end{equation*}
To obtain the corresponding inequality for $\psi$, we use \cite[Corollary 5.1]{BKLNW_21}, which states that for all $x\geq e^{40}$,
\begin{equation*}
    0<\psi(x)-\theta(x)<a_1x^{1/2}+a_2x^{1/3},
\end{equation*}
with $a_1=1+1.93378\cdot 10^{-8}$ and $a_2=1.0432$. Thus,
\begin{align*}
    \sum_{n\in I(x,u)}\Lambda(n)&\leq \psi(x_+)-\psi(x_-) +\log x\notag\\
    &<c_0ux\cdot \mathcal{E}_1(x,T)+a_1\sqrt{x_+}+a_2\sqrt[3]{x_+}+\log x\\
    &\leq c_0ux\cdot\mathcal{E}_1(x,T)+\mathcal{E}_2(x),
\end{align*}
where
\begin{equation*}
    \mathcal{E}_2(x) = a_1(x+(c_0-1)\lambda x)^{1/2} + a_2(x+(c_0-1)\lambda x)^{1/3} + \log x.
\end{equation*}
Substituting this into \eqref{perronbound2}, we obtain
\begin{align}\label{secondsumbound}
    \frac{e^{\kappa\lambda}}{\pi T}\int^{\lambda}_{\theta'/T}\frac{1}{u^2}\sum_{I(x,u)}\Lambda(n)\mathrm{d}u &\leq \frac{e^{\kappa\lambda}}{\pi T}\int_{\theta'/T}^{\lambda}\left(\frac{c_0 x\cdot \mathcal{E}_1(x,T)}{u}+\frac{\mathcal{E}_2(x)}{u^2}\right)\mathrm{d}u\notag\\
    &\leq\frac{e^{\kappa\lambda}}{\pi}\left(\frac{c_0 x\cdot \mathcal{E}_1(x,T)}{T}\log\left(\frac{\lambda}{\theta'}T\right)+\frac{\mathcal{E}_2(x)}{\theta'}\right).
\end{align}
Combining \eqref{newwolkebound}, \eqref{firstsumbound} and \eqref{secondsumbound}, the proof of Theorem~\ref{mainthm} is complete upon optimising over $\lambda$ to compute values of $x_M$, $\alpha\in(0,1/2]$ and $M$ such that
\begin{equation*}
    \psi(x)=x-\sum_{\substack{\rho=\beta+i\gamma\\|\gamma|\leq T}}\frac{x^\rho}{\rho}+O^*\left(\frac{Mx\log x}{T}\right),
\end{equation*}
for all $x\geq x_M$ and $\max\{51,\log x\}<T<\frac{x^{\alpha}}{2}-1$. Note we cannot take $\alpha>1/2$, as $\mathcal{E}_2(x)$ would not be $O\left(x\log x/T\right)$ for values of $T$ asymptotically larger than $\sqrt{x}$.

Under assumption of the Riemann hypothesis, we also have the following.
\begin{theorem}\label{riemannthm2}
    Assuming the Riemann hypothesis, we can take $M=4.150$ in \eqref{oureq} with $\log x\geq 1000$ and $\alpha=1/2$. 
\end{theorem}
\begin{proof}
    Calculate $M$ as above using $\lambda=0.52$ and the value of $K=0.6373$ from Theorem~\ref{riemannthm1}.
\end{proof}

\section{Application: primes between consecutive powers}\label{powersect}

For sufficiently large $x$, the best unconditional short-interval result for primes is $[x,x+x^{0.525}]$, from Baker, Harman, and Pintz \cite{B_H_P_2001}. This can be narrowed to $(x,x+C\sqrt{x}\log x]$ for some constant $C$ if assuming the Riemann hypothesis, as per Cram{\'e}r's result \cite{cramer1936order}. The latest explicit version of this is \cite[Thm.~1.5]{C_M_S_19}, with $C=22/25$ for all $x\geq 4$. Legendre conjectured that something just better than this should be true: that there should be a prime between $n^2$ and $(n+1)^2$ for all positive integers $n$. This is approximately equivalent to primes in intervals of the form $(x,x+2\sqrt{x}]$. Although proving Legendre's conjecture is out of reach even under RH, we do know there are primes between higher consecutive powers. Using Ingham's method from \cite{Ingham_37}, it was shown in \cite{cully2023primes} that there are primes between consecutive cubes, $n^3$ and $(n+1)^3$, for all $n\geq \exp(\exp(32.892))$.

It is also possible to find primes between higher consecutive powers for all positive $n$. In \cite{cully2023primes}, this is done using an explicit version of Goldston's estimate estimate for the error in the prime number theorem \cite{Goldston_83}. We can now use Theorem~\ref{mainthm} in place of Goldston's result, and prove Theorem~\ref{powerthm}.

To consider the interval $(n^m, (n+1)^m)$, we set $n = x^\frac{1}{m}$ and look at the slightly smaller interval $(x,x+h]$ with $h=mx^{1-1/m}$. Theorem~\ref{mainthm} implies
\begin{equation} \label{psipsi} 
    \psi(x+h) - \psi(x) \geq \ h - \sum_{\substack{\rho=\beta+i\gamma\\|\gamma|\leq T}} \left|\frac{(x+h)^\rho-x^\rho}{\rho}\right| - M \frac{G(x,h)}{T}
\end{equation}
for all $x\geq x_M$ and $\max\{51,\log x\}<T<(x^{\alpha}-2)/2$, where $G(x,h) = (x+h)\log(x+h) + x\log x$, and values for $x_M$, $\alpha$ and $M$ are given in Table \ref{maintable}. In order to find at least one prime in $(x,x+h]$, (\ref{psipsi}) needs to be positive. Fixing an $m$, we want to maximise (\ref{psipsi}) with respect to $T$, and solve for $x$. This process follows the same steps as in Section 4 of \cite{cully2023primes}, and we set $T = x^\mu$ for some $\mu \in (0,1)$.\footnote{Note that $\mu$ here is denoted $\alpha$ in \cite{cully2023primes}.} We then arrive at a similar condition to equation (21) in \cite{cully2023primes}: that there are primes between consecutive $m^{\text{th}}$ powers for all $x$ satisfying
\begin{equation} \label{condition_psi} 
    1 - F(x) - M \frac{G(x,h)}{x^\mu h} + \frac{E(x)}{h} > 0
\end{equation}
where $F(x)$ and $E(x)$ are defined as in equation (19) and (20), respectively, of \cite{cully2023primes}. A zero-free region is used in this definition of $F(x)$, so we use the $\nu(t)$ defined in Section \ref{zfsect} of this paper, instead of that used in \cite{cully2023primes}.

For $\log x_M = 10^3$ we can take $M = 0.6651$, so for $m=140$ we can take $\mu = 0.0080155$ to have (\ref{condition_psi}) hold for all $\log x\geq 4242$. Note that with these values the condition on $T$ in Theorem~\ref{mainthm} is satisfied for all $\log x\geq 375$. The interval estimates for primes in \cite{CH_L_arX} (corrected from the original paper \cite{CH_L_21}) for $x\geq 4\cdot 10^{18}$ and $x\geq e^{600}$ verify that there are primes between consecutive $140^{\text{th}}$ powers for $\log(4\cdot 10^{18}) \leq \log x \leq 4367$. The calculations of \cite{O_H_P_14} verify the remaining smaller values of $x$.


\section{Application: the error term in the prime number theorem}\label{errorsect}
The prime number theorem is equivalent to the statement $\psi(x)\sim x$. For large $x$, the best unconditional estimates on the error $|\psi(x)-x|$ are from Platt and Trudgian \cite[Theorem 1]{P_T_2021}. For $R=5.5734125$ they give values for $X$, $A$, $B$ and $C$ such that
\begin{equation}\label{timdaveeq}
    \left|\frac{\psi(x)-x}{x}\right|\leq A\left(\frac{\log x}{R}\right)^B\exp\left(-C\sqrt{\frac{\log x}{R}}\right)
\end{equation}
for all $\log x\geq X$. To obtain \eqref{timdaveeq}, Platt and Trudgian employ a method of Pintz \cite{Pintz_80} and use Dudek's error term for the Riemann--von Mangoldt formula (Theorem~\ref{dudekthm}).

In this section we prove Theorem~\ref{errorthm}, which is a non-trivial improvement on Platt and Trudgian's result. Most of the improvement comes from using Theorem~\ref{mainthm} in place of Dudek's error term. We will also incorporate some other recent results. 

Firstly, we use the smaller value $R=5.5666305$ which is the same as $R_0$ appearing in the zero-free region in see Lemma \ref{classlem}. In the following lemma, we also make small improvements to some of the zero-density estimates in \cite{K_L_N_2018}.

\begin{lemma}\label{densitylemma}
    Let $N(\sigma,T)$ be the number of zeros $\rho=\beta+i\gamma$ of the Riemann zeta-function with $\sigma<\beta\leq 1$ and $0\leq\gamma\leq T$. Then, for the values of $C_1(\sigma)$ and $C_2(\sigma)$ in Table \ref{densitytable}, we have
    \begin{equation*}
        N(\sigma,T)\leq C_1(\sigma)T^{8(1-\sigma)/3}(\log T)^{5-2\sigma}+C_2(\sigma)\log^2T.
    \end{equation*}
\end{lemma}
\begin{proof}
    Using Platt and Trudgian's verification of the Riemann hypothesis up to $H_0=3\cdot 10^{12}$ \cite{P_T-RH_21}, and the divisor function estimate in Theorem 2 of \cite{C_T_19} to replace (3.13) of \cite{K_L_N_2018}, we can recalculate the constants in Lemma 4.14 of \cite{K_L_N_2018}. Using the notation from \cite{K_L_N_2018}, we want to optimise over $k$, $\mu$, $\alpha$, $\delta$, $d$, $H$, and $\eta$. We chose $H = H_0 - 1$, $\eta=0.2535$, $k=1$, $\mu = 1.237$, $\delta = 0.313$ and optimised over the other parameters for each $\sigma$.
\end{proof}

\def\arraystretch{1.5}
\begin{table}[h]
\centering
\caption{Some values of $C_1(\sigma)$ and $C_2(\sigma)$ for Lemma \ref{densitylemma}.}
\begin{tabular}{|c|c|c|c|c|c|c|c|}
\hline
$\sigma$ & $\alpha$ & $d$ & $C_1(\sigma)$ & $C_2(\sigma)$ \\
\hline
0.980 & 0.063 & 0.336 & 15.743 & 2.214 \\
\hline
0.982 & 0.063 & 0.336 & 15.878 & 2.204 \\
\hline
0.984 & 0.061 & 0.336 & 16.013 & 2.187 \\
\hline
0.986 & 0.061 & 0.336 & 16.148 & 2.171 \\
\hline
0.988 & 0.060 & 0.337 & 16.284 & 2.148 \\
\hline
0.990 & 0.060 & 0.337 & 16.421 & 2.132\\
\hline
0.992 & 0.058 & 0.337 & 16.558 & 2.115\\
\hline
\end{tabular}
\label{densitytable}
\end{table}

\begin{proof}[Proof of Theorem~\ref{errorthm}]
    We only require a slight modification of the proof of \cite[Theorem 1]{P_T_2021}. First we assume $\log x\geq 1000$ so that we may take $x_M=\exp(1000)$, $\alpha=1/10$ and $M=2.045$ in Theorem~\ref{mainthm}. The argument in \cite[pp. 874--875]{P_T_2021} then follows through with minimal modification so that if $C_1(\sigma)$ and $C_2(\sigma)$ are as in Lemma \ref{densitylemma} and
    \begin{align*}
        k(\sigma,x_0)&=\left[\exp\left(\frac{10-16\sigma}{3}\sqrt{\frac{\log x_0}{R}}  \right)\left(\sqrt{\frac{\log x_0}{R}}\right)^{5-2\sigma}\right]^{-1}\\
        C_3(\sigma,x_0)&=M(\log x_0)\exp\left(-2\sqrt{\frac{\log x_0}{R}}\right)k(\sigma,x_0)\\
        C_4(\sigma,x_0)&=x_0^{\sigma-1}\left(\frac{2}{\pi}\frac{\log x_0}{R}+1.8642\right)k(\sigma,x_0)\\
        C_5(\sigma,x_0)&=8.01\cdot C_2(\sigma)\exp\left(-2\sqrt{\frac{\log x_0}{R}}\right)\frac{\log x_0}{R}k(\sigma,x_0)\\
        A(\sigma,x_0)&=2.0025\cdot 2^{5-2\sigma}\cdot C_1(\sigma)+C_3(\sigma,x_0)+C_4(\sigma,x_0)+C_5(\sigma,x_0),
    \end{align*}
    then
    \begin{equation*}
        \left|\frac{\psi(x)-x}{x}\right|\leq A(\sigma,x_0)\left(\frac{\log x}{R}\right)^{\frac{5-2\sigma}{2}}\exp\left(\frac{10-16\sigma}{3}\sqrt{\frac{\log x}{R}}\right)
    \end{equation*}
    for all $\sigma\in[0.75,1)$ and $x>x_0$ provided $A(\sigma,x)$ is decreasing for $x>x_0$. The only difference between our formulae and those on page 875 of \cite{P_T_2021} is the expression for $C_3(\sigma,x_0)$. Taking $x_0=\exp(X)$ it is then possible to compute the values of $A$ in Table \ref{errortable}. The values of $\sigma$ in Table \ref{errortable} were chosen to optimise the value of $\epsilon_0$.
\end{proof}
\begin{remark}
    Using Theorem~\ref{mainthm} meant that the $C_3(\sigma,x_0)$ term in the above proof became essentially negligible for the values of $x_0$ considered. Thus, reducing $M$ is unlikely to substantially improve Theorem~\ref{errorthm} using the above method.
\end{remark}

We also have the following corollary (cf.\ \cite[Corollary 1]{P_T_2021}).
\begin{corollary}\label{errorcor}
    For each row $\{X,A,B,C\}$ in Table \ref{errortable} we have
    \begin{equation*}
        \left|\frac{\theta(x)-x}{x}\right|\leq A_1\left(\frac{\log x}{R}\right)^B\exp\left(-C\sqrt{\frac{\log x}{R}}\right),\quad\text{for all}\ \log x\geq X 
    \end{equation*}
    where $A_1=A+0.1$.
\end{corollary}
\begin{proof}
     By \cite[Corollary 5.1]{BKLNW_21}, we have for $x\geq\exp(1000)$ that
    \begin{align*}
        \psi(x)-\theta(x)<a_1x^{1/2}+a_2x^{1/3},
    \end{align*}
    with $a_1=1+1.99986\cdot 10^{-12}$ and $a_2=1+1.936\cdot 10^{-8}$. The result then follows by a straightforward computation since
    \begin{equation*}
        \left|\frac{\theta(x)-x}{x}\right|\leq\frac{\psi(x)-\theta(x)}{x}+\left|\frac{\psi(x)-x}{x}\right|.\qedhere
    \end{equation*}
\end{proof}

Corollary \ref{errorcor} allows us to improve on existing knowledge of an inequality due to Ramanujan. Namely, in one of his notebooks, Ramanujan proved that
\begin{equation}\label{ramaneq}
    \pi(x)^2<\frac{ex}{\log x}\pi\left(\frac{x}{e}\right)
\end{equation}
holds for sufficiently large $x$ \cite[pp. 112--114]{berndt2012ramanujan}.

It is still an open problem to determine the largest value of $x$ for which \eqref{ramaneq} holds. However, it is widely believed that the last integer counterexample occurs at $x=38,358,837,682$. In fact, this follows under assumption of the Riemann hypothesis \cite[Theorem 1.3]{dudek2015solving}. 

Using their expression for $|\theta(x)-x|$, Platt and Trudgian \cite[Theorem 2]{P_T_2021} were able to show (unconditionally) that \eqref{ramaneq} holds for $x\geq\exp(3915)$. Substituting our results for Corollary \ref{errorcor} when $X=3600$ into the formulae on page 879 of \cite{P_T_2021} gives a small improvement. In particular, we get that Ramanujan's inequality \eqref{ramaneq} holds for $x\geq\exp(3604)$ (Corollary \ref{ramancor}).

We also note that the second author recently showed that Ramanujan's inequality \eqref{ramaneq} holds for $38,358,837,683\leq x\leq\exp(103)$ \cite[Theorem 5.1]{johnston2022improving}. However, we are not able to improve on this result here as Theorem~\ref{errorthm} and Corollary \ref{errorcor} only give good bounds on prime counting functions for large $x\geq\exp(1000)$.

\section*{Acknowledgements}
Thanks to all our colleagues at UNSW Canberra. Particularly our supervisor Tim Trudgian for his constant support, Aleks Simoni\v{c} for helping us with those pesky semicircles, and Ethan Lee for the long discussions over tea. We also thank Ramar{\'e} for his insights and correspondence.

\newpage
\section*{Appendix: Tables \ref{theorem2table1} and \ref{maintable}}
\def\arraystretch{1.5}
\begin{table}[h]
\centering
\caption{Some corresponding values of $x_K$, $\alpha$, $\omega$, $D$ and $K$ for Theorem~\ref{wolkeprop}.}

\begin{tabular}{|c|c|c|c|c|c|}
\hline
$\log(x_K)$ & $\alpha$ & $\omega$ & $\overline{\omega}$ & $D$ & $K$\\
\hline
$40$ & $1/2$ & $0$ & $0$ & --- & $2.053$\\
\hline
$10^3$ & $1/2$ & $0$ & $0$ & --- & $1.673$\\
\hline
$10^{10}$ & $1/2$ & $0.3$ & $0.3$ & $0.54$ & $3.191$\\
\hline
$10^{13}$ & $1/2$ & $1$ & $1.4$ & $0.50$ & $0.6367$\\
\hline
$10^3$ & $1/10$ & $0.2$ & $0.2$ & $0.45$ & $2.596$\\
\hline
$10^{10}$ & $1/10$ & $0.9$ & $0.9$ & $0.54$ & $11.77$\\
\hline
$10^3$ & $1/100$ & $0.8$ & $0.8$ & $0.52$ & $2.186$\\
\hline
$10^{10}$ & $1/100$ & $1$ & $1.5$ & $0.50$ & $0.6367$\\
\hline
\end{tabular}
\label{theorem2table1}
\end{table}

\def\arraystretch{1.5}
\begin{table}[h]
\centering
\caption{Some corresponding values of $x_M$, $\alpha$, $\lambda$ and $M$ for Theorem~\ref{mainthm}.}
\begin{tabular}{|c|c|c|c|}
\hline
$\log(x_M)$ & $\alpha$ & $\lambda$ & $M$\\
\hline
$40$ & $1/2$ & $0.48$ & $6.431$\\
\hline
$10^3$ & $1/2$ & $0.52$ & $5.823$\\
\hline
$10^{10}$ & $1/2$ & $0.52$ & $4.143$\\
\hline
$10^{13}$ & $1/2$ & $0.52$ & $4.140$\\
\hline
$10^3$ & $1/10$ & $1.05$ & $2.045$\\
\hline
$10^{10}$ & $1/10$ & $1.06$ & $1.384$\\
\hline
$10^3$ & $1/100$ & $1.80$ & $0.6651$\\
\hline
$10^{10}$ & $1/100$ & $1.88$ & $0.6269$\\
\hline
\end{tabular}
\label{maintable}
\end{table}

Note that although we could have considered larger values of $\alpha\in(0,1]$ in Table \ref{theorem2table1}, the restriction $\alpha\leq 1/2$ is required for Theorem~\ref{mainthm}.

\newpage

\printbibliography

\end{document}